\documentclass{article}

\usepackage[a4paper,
 total={140mm,237mm},
 left=35mm,
 top=30mm,
 ]{geometry}
\usepackage[utf8]{inputenc}
\usepackage{mathpazo}
\usepackage{hyperref}
\usepackage[anythingbreaks]{breakurl}
\usepackage{doi}

\usepackage[T1]{fontenc}
\usepackage{amssymb, amsmath}
\usepackage[dvipsnames]{xcolor}
\usepackage[english]{babel}
\usepackage{amsthm}
\usepackage{enumitem}
\usepackage{tikz}
\usepackage{bm}
\usepackage{csquotes}
\usepackage{extdash}

\usepackage[format=plain,
            labelfont=it,
            textfont=it]{caption}
\usepgflibrary{shapes.geometric}
\usepackage{float}
\allowdisplaybreaks

 \hypersetup{breaklinks=true, 
           colorlinks=true,   
        }

\usepackage[maxnames=20]{biblatex} 
\bibliography{Ergodic.bib}

\usepackage{subfiles}

\AtBeginBibliography{\small}
 
\newtheorem{theorem}{Theorem}[section]
\newtheorem{innercustomthm}{Theorem}
\newenvironment{customthm}[1]
  {\renewcommand\theinnercustomthm{#1}\innercustomthm}
  {\endinnercustomthm}
\newtheorem{corollary}[theorem]{Corollary}
\newtheorem{lemma}[theorem]{Lemma}
\newtheorem{prop}[theorem]{Proposition}

\theoremstyle{remark}
\newtheorem{remark}[theorem]{Remark}

\theoremstyle{definition}

\newtheorem{definition}[theorem]{Definition}
\newtheorem{question}[theorem]{Question}

\title{Invariant Keisler measures for $\omega$- categorical structures}
\author{Paolo Marimon}

\newcommand{\indep}[2]{%
  \mathrel{
    \mathop{
      \vcenter{
        \hbox{\oalign{\noalign{\kern-.3ex}\hfil$\vert$\rlap{$^\mathrm{#2}$}\hfil\cr
              \noalign{\kern-.7ex}
              $\smile$\cr\noalign{\kern-.3ex}}}
      }
    }\displaylimits_{#1}
  }
}

\newcommand{\notindep}[2]{%
  \mathrel{
    \mathop{
      \vcenter{
        \hbox{\oalign{\noalign{\kern-.3ex}\hfil$\not\vert$\rlap{$^\mathrm{#2}$}\hfil\cr
              \noalign{\kern-.7ex}
              $\smile$\cr\noalign{\kern-.3ex}}}
      }
    }\displaylimits_{#1}
  }
}

\begin{document}

\maketitle
\begin{abstract}
A recent article of Chernikov, Hrushovski, Kruckman, Krupinski, Moconja, Pillay and Ramsey finds the first examples of simple structures with formulas which do not fork over the empty set but are universally measure zero. In this article we give the first known simple $\omega$-categorical counterexamples. These happen to be various $\omega$-categorical Hrushovski constructions. Using a probabilistic independence theorem from Jahel and Tsankov, we show how simple $\omega$-categorical structures where a formula forks over $\emptyset$ if and only if it is universally measure zero must satisfy a stronger version of the independence theorem. 
\end{abstract}

\section{Introduction}

Keisler measures yield a natural notion of smallness for a definable subset of a structure: a subset $X$ of $M^x$  defined by $\phi(x,a)$ is \textbf{universally measure zero} if it is assigned measure zero by every invariant Keisler measure on $M$. A more classical model theoretic notion of smallness is forking: dividing captures the idea that a small subset of a model can be "moved enough" by automorphisms so as to not overlap with itself.\\

In stable theories, a definable set is universally measure zero if and only if it forks over $\emptyset$ \cite{lots_of_authors}. This is also the case in $\omega$\Hyphdash* categorical NIP theories \cite{SamI}. Until recently, it was unknown whether the two notions coincided in simple theories. The first counterexample, showing that in simple theories there can be non-forking formulas which are universally measure zero is given in \cite{lots_of_authors}. Using the same technique, the authors also give examples of simple groups which are not definably amenable and prove some positive results in the context of small theories. Neither of the counterexamples given is $\omega$\Hyphdash* categorical. The theory of the first counterexample contains as a reduct the theory of the free action of the free group $\mathrm{F}_5$ on an infinite set, which has infinitely many $1$-types over a single parameter. Meanwhile, on the definably amenable group case, all groups definable in an $\omega$\Hyphdash* categorical simple theory are definably amenable \cite[Corollary 4.14]{lots_of_authors}. Indeed, $\omega$\Hyphdash* categorical supersimple groups are actually amenable being finite-by-abelian-by-finite \cite{WagEv}. 
\\

It is natural to ask whether adding the assumption of $\omega$\Hyphdash* categoricity we can prove that being universally measure zero is the same as forking. Firstly, the known counterexample makes heavy use of a construction which is inherently not $\omega$\Hyphdash* categorical. Secondly, $\omega$\Hyphdash* categoricity implies that any invariant Keisler measure is also definable, i.e. the set of parameters $a$ for which $\mu(\phi(x,a))=\alpha$ is $\emptyset$-definable. In general, this is a substantially stronger assumption than invariance.\\ 

Another motivation for an $\omega$\Hyphdash* categorical counterexample comes from the study of $\omega$\Hyphdash* categorical $MS$-measurable structures \cite{MS}. Until recently \cite{Measam, me}, it was an open question whether all supersimple $\omega$\Hyphdash* categorical structures of finite $SU$-rank are $MS$-measurable. Indeed, this question of Elwes and Macpherson \cite{EM} was the initial motivation for our study of Keisler measures in the context of $\omega$\Hyphdash* categorical Hrushovski constructions. An $MS$-measurable structure \cite{MS} has a dimension-measure function which is definable, finite and such that the dimension and measures satisfy Fubini's theorem. In \cite{me}, we showed that in an $\omega$\Hyphdash* categorical $MS$-measurable structure we can always take the dimension part of the dimension-measure to be $SU$-rank. With this observation, we can see that if an $\omega$\Hyphdash* categorical structure is $MS$-measurable, then a formula forks over the $\emptyset$ if and only if it is universally measure zero.\\

In this article, we show how for various classes of supersimple $\omega$\Hyphdash* categorical Hrushovski constructions we have formulas that do not fork over $\emptyset$ but are universally measure zero:
\begin{customthm}{\ref{mainthm}}
There are $\omega$\Hyphdash* categorical supersimple theories $T$ of finite $SU$-rank with a formula $\phi(x, a)$ which does not fork over the empty set, but which is universally measure zero. 
\end{customthm}
More generally, in Theorem \ref{strongITT} we show that if forking and being universally measure zero agree in a simple $\omega$\Hyphdash* categorical structure, then it must satisfy a stronger version of the independence theorem. It is easy to build $\omega$\Hyphdash* categorical Hrushovski constructions for which this fails. These structures are  extremely amenable in the sense of \cite{FOAmen}, which implies the existence of invariant types, and so invariant Keisler measures, in each variable.\\

We begin with Section \ref{K&E}, where we introduce ergodic measures and Keisler measures. Ergodic measures simplify our study since any measure can be considered as an "integral average" of them. Meanwhile, Keisler measures are the natural notion of a measure on a first-order structure. In Section \ref{WAI}, we study the $L^2$-spaces associated with an invariant Keisler measure. From some results of Jahel and Tsankov \cite[Theorem 3.2, Corollary 3.5]{JahelT}, we know that in $\omega$\Hyphdash* categorical structures, ergodic measures are better behaved and a weak form of algebraic independence implies a form of probabilistic independence in the measure (Corollary \ref{probind}). In Section \ref{SIT}, we show how in simple $\omega$\Hyphdash* categorical structures with forking and being universally measure zero agreeing we have a stronger version of the independence theorem (Theorem \ref{strongITT}). We also draw some implications for $\omega$\Hyphdash* categorical $MS$-measurable structures. Finally, in Section \ref{counterex}, we conclude giving the example of an $\omega$\Hyphdash* categorical structure, supersimple of finite $SU$-rank with a formula which does not fork over the empty set but which is universally measure zero.\\

This article requires some knowledge in model theory. Chapters 1-4 of Tent \& Ziegler's book \cite{TZ} should be sufficient, together with Ch.16 of \cite{Poizat} for understanding imaginaries and weak elimination of imaginaries. On simple theories, Chapters 2 and 3 of \cite{Kimsimp} cover the relevant definitions and results, including the definition of $SU$-rank and the independence theorem. We also require some basic knowledge of Hilbert spaces and $L^2$-spaces, such as Chapter 1 of \cite{Conway}. In Section \ref{K&E} we give a self-contained introduction to ergodic measures and Keisler measures. The reader interested in $MS$-measurable structures should consult \cite{MS} or \cite{EM}. This article does not require knowledge of Hrushovski constructions since all of the properties that we use are listed in Theorem \ref{constr}. Further details about $\omega$\Hyphdash* categorical Hrushovski constructions can be found in \cite[\S 6.2.1]{Wagner:ST}, while \cite{me}, especially in the appendix, provides the details on the specific construction we give as an example in Theorem \ref{mainthm}.\\

To conclude, we provide some notation and conventions. Firstly, we work with a complete countable first-order $\mathcal{L}$-theory $T$. From Section \ref{WAI}, by $\mathcal{M}$ we denote the countable model of an $\omega$\Hyphdash* categorical theory. We write $M$ for when we consider $\mathcal{M}$ as a set. We use greek letters $\phi, \psi, \chi, \dots$ to refer to formulas. The letters $x,y,z, \dots $ indicate variables, and may also indicate a finite tuple of variables. Similarly, the lowercase letters $a,b,c, \dots$ indicate parameters from $M$, and may also indicate a finite tuple. Meanwhile, we indicate subsets of $M$ by the uppercase $A, B, C, \dots$ For $a, a'$ tuples from $\mathcal{M}$ and $A\subseteq M$, we write $a\equiv_A a'$ to say that $a$ and $a'$ have the same type over $A$.\\

\textbf{Acknowledgements:} This article was written during my PhD at Imperial College London under the supervision of David Evans and Charlotte Kestner, both of whom I thank for their support and help. Nicholas Ramsey gave an excellent talk at Imperial about \cite{lots_of_authors}, which motivated me to look at whether $\omega$\Hyphdash* categorical Hrushovski constructions may also have been a counterexample. I am also grateful to my friend Matteo Tabaro for helpful discussions during my learning of ergodic theory. Finally, some of the ideas behind this article are greatly indebted to Ehud Hrushovski and his unpublished article \cite{AER}, which he kindly shared with me. 
 
\section{Keisler measures and ergodic measures}\label{K&E}

We begin with a brief and self-contained introduction to ergodic measures and Keisler measures. Firstly, we introduce ergodic measures in a general context. Then, we explain their utility for the study of Keisler measures. 

\subsection{Ergodic measures}

Ergodic measures are an essential tool in our paper. They are better behaved than invariant measures and any invariant measure can be decomposed as an integral average of ergodic measures. Here we briefly introduce these measures and mention some basic results about them. Chapter 12 of Phelps' book \cite{Choquet} covers the theory we discuss at the adequate level of generality for our purposes.\\

We work in the context of a topological group $G$ acting on a topological space $X$ via a continuous action $\cdot:G\times X\to X$. When $X$ is compact and Hausdorff, we call this action a $\pmb{G}$\textbf{-flow}. Let $\mathcal{B}(X)$ be the set of Borel subsets of $X$, and let $\mu:\mathcal{B}(X)\to [0,1]$ be a Borel probability measure. We say that $\mu$ is $\pmb{G}$\textbf{-invariant} if for any $\tau\in G$ and any $A\in\mathcal{B}(X)$, we have that $\tau\cdot A\in \mathcal{B}(X)$ and $\mu(\tau\cdot A)=\mu(A)$. 

\begin{definition}
We say that a $G$-invariant Borel probability measure $\mu$ is $\pmb{G}$-\textbf{ergodic} if for all $A\in\mathcal{B}(X)$, we have that if for all $\tau\in G$, 
\[\mu(A	\bigtriangleup \tau \cdot A)=0, \]
then either $\mu(A)=0$ or $\mu(A)=1$.
\end{definition}

\noindent An alternative definition tells us that any invariant function is constant \cite[Prop.2.14]{EinsW}:

\begin{prop}\label{erginv} Let $(X, \mathcal{B}, \mu)$ be a probability space and $G$ be a group acting on $X$ such that $\mu$ is $G$-invariant. Then, the following are equivalent:
\begin{enumerate}
    \item the measure $\mu$ is $G$-ergodic; and
    \item any measurable function $f:X\to\mathbb{C}$, which is $G$-invariant almost everywhere (i.e. for any $\tau\in G, f\circ \tau =f$ a.e.) is constant almost everywhere. 
\end{enumerate}
\end{prop}

We are interested in studying $G$-ergodic measures since, when $X$ is metrizable, they yield an ergodic decomposition of any $G$-invariant measure. Hence, their study is essential to the understanding of the $G$-invariant measures on $X$. Below, for $f\in C(X, \mathbb{C})$, we write $\mu(f)=\int_X f \mathrm{d}\mu$. From \cite[p.77]{Choquet} we have:

\begin{theorem}[Ergodic decomposition] \label{ergdecomp} Let $X$ be a compact metrizable space with a group $G$ acting continuously on it. Let $\mu$ be a $G$-invariant Borel probability measure on $X$. Then, the space 
$\mathfrak{M}(X)$ of $G$-invariant Borel probability  measures on $X$ is also metrizable and set of $G$-ergodic measures $\mathrm{Erg}(X)$ is a Borel subset of $\mathfrak{M}(X)$. Furthermore, there is a unique Borel probability measure $\mathfrak{m}$ on $\mathfrak{M}(X)$ such that for any $f\in C(X, \mathbb{C})$, 
\[\mu(f)=\int_{\mathrm{Erg}(X)} \nu(f) \mathrm{d}\mathfrak{m}(\nu).\]
\end{theorem}

\subsection{Keisler measures}

Keisler measures are finitely additive probability measures on the space of definable subsets of a structure. Chapter 7 of Pierre Simon's book \cite{NIP} is a good introduction to the subject. We give a brief self-contained discussion and note the importance of ergodic measures in the study of an invariant Keisler measures.\\

Let $T$ be a complete $\mathcal{L}$-theory, where $\mathcal{L}$ is a countable first-order language and $\mathcal{M}\models T$. Let $\mathrm{Def}_x(M)$ denote the Boolean algebra of $\mathcal{L}(M)$-formulas in the free variable $x$ up to $\mathrm{Th}(\mathcal{M}_M)$-equivalence, where $\mathcal{M}_M$ is the expansion of $\mathcal{M}$ by constant symbols for each element of $M$. Let $S_x(M)$ be the Stone space of types over $M$ in the variable $x$ with the usual topology. We write $[\phi(x,a)]$ for the clopen set of types containing the formula $\phi(x,a)$.

This space is always compact and Hausdorff. Moreover, when $M$ is countable, it is also metrizable. This can be seen by taking an enumeration of the $\mathcal{L}(M)$-formulas $\phi_1, \phi_2, \dots$ and then for $p_1, p_2\in S_x(M),$ letting $d(p_1,p_2)=\frac{1}{2^n}$, where $n$ is the smallest natural number such that $\phi_n$ is not contained in both $p_1$ and $p_2$. The automorphism group $\mathrm{Aut}(M)$ naturally acts on $\mathrm{Def}_x(M)$ and on $S_x(M)$, where for $\tau\in \mathrm{Aut}(M)$, $\tau\cdot \phi(x, a)=\phi(x, \tau(a))$ and for $p\in S_x(M)$, 
\[\tau\cdot p=\{\phi(x, \tau(a)) :  \phi(x, a)\in p\}.\]

\begin{definition} Let $\mathcal{M}$ be an $\mathcal{L}$-structure. We say that $\mu:\mathrm{Def}_x(M)\to [0,1]$ is a \textbf{Keisler measure} if it is a finitely additive measure such that $\mu(x=x)=1$. We say that $\mu$ is $\mathrm{Aut}(M)$-\textbf{invariant} if for any $\tau\in \mathrm{Aut}(M)$,
\[\mu(\tau\cdot \phi(x, a))=\mu(\phi(x, a)).\]
If $\mathcal{M}$ is strongly $\omega$-homogeneous, this is equivalent to saying that, if $a\equiv a'$, then
$\mu(\phi(x, a))=\mu(\phi(x, a'))$.
\end{definition}

As noted in the introduction, Keisler measures give us a natural notion of smallness for a definable set.

\begin{definition} Let $\mathbb{M}$ be a monster model and $A\subset \mathbb{M}$ a small subset. We say that $\phi(x,a)\in\mathrm{Def}_x(\mathbb{M})$ is \textbf{universally measure zero} over $A$ if $\mu(\phi(x,a))=0$ for any $\mathrm{Aut}(\mathbb{M}/A)$-invariant Keisler measure. We call $\mathcal{O}(A)$ the set of formulas which are universally measure zero over $A$. If $A=\emptyset$, we say that $\phi(x,a)$ is universally measure zero.
\end{definition}

For $A\subseteq \mathbb{M}$ small, $\mathcal{O}_x(A)=\mathcal{O}(A)\cap \mathrm{Def}_x(\mathbb{M})$ forms an ideal  in the Boolean algebra of $\mathrm{Def}_x(\mathbb{M})$ \cite{lots_of_authors}. Similarly, $F_x(A)$, the set of formulas in the variable $x$ with parameters from $\mathbb{M}$ forking over $A$, also forms an ideal. Our main result is that there are $\omega$\Hyphdash* categorical supersimple theories for which $\mathcal{O}(\emptyset)$ strictly contains the set of formulas forking over $\emptyset$, $F(\emptyset)$. 

\begin{remark}\label{noloss} In this article we mainly study $\mathrm{Aut}(M)$-invariant measures on a countable $\omega$\Hyphdash* categorical structure $\mathcal{M}$. Any $\mathrm{Aut}(\mathbb{M})$-invariant measure on $\mathbb{M}$ induces an $\mathrm{Aut}(M)$-invariant measure on any small model $\mathcal{M}\prec\mathbb{M}$. Hence, if for $a\in M$, $\mu(\phi(x,a))=0$ for any $\mathrm{Aut}(M)$-invariant Keisler measure on $M$, then, $\phi(x,a)$ is universally measure zero. In particular, for a small theory $T$, we have a countable $\omega$-saturated model $\mathcal{N}$, and so any $\mathrm{Aut}(\mathbb{M})$-invariant Keisler measure $\mu$ on $\mathbb{M}$ is entirely determined by its restriction to formulas with parameters from $N$. By $\omega$-saturation, working in an $\omega$\Hyphdash* categorical theory, no generality is lost 
by looking at the $\mathrm{Aut}(M)$-invariant Keisler measures for the countable model $\mathcal{M}$ instead of the $\mathrm{Aut}(\mathbb{M})$-invariant Keisler measures for the monster model $\mathbb{M}$.
\end{remark}


A Keisler measure $\mu$ can be extended uniquely to a regular $\sigma$-additive Borel probability measure $\mu: \mathcal{B}_x(M)\to[0,1]$, where $\mathcal{B}_x(M)$ is the set of Borel subsets of $S_x(M)$ \cite{NIP}[$\S 7.1$]. Conversely, any regular Borel probability measure on $S_x(M)$ induces a Keisler measure by considering its restriction to clopen sets. This correspondence still holds between $\mathrm{Aut}(M)$-invariant Keisler measures in the variable $x$ and $\mathrm{Aut}(M)$-invariant regular Borel probability measures on $S_x(M)$. From here, we shall speak interchangeably of the two. 

\begin{remark}
An $\mathrm{Aut}(M^{eq})$-invariant Keisler measure $\mu^{eq}$ on $\mathcal{M}^{eq}$ in the real variable $x$ is entirely determined by its restriction to its induced $\mathrm{Aut}(M)$-invariant Keisler measure $\mu$ on $\mathcal{M}$. Furthermore, $\mu^{eq}$ is ergodic if and only if $\mu$ is.
\end{remark}

If we are interested in studying the universally measure zero formulas for a structure, it is helpful to study $\mathrm{Aut}(M)$-ergodic measures on $S_x(M)$. From the ergodic decomposition \ref{ergdecomp}, we get:

\begin{corollary}\label{decomp} Let $\mathcal{M}$ be a countable structure and let $\mu$ be an $\mathrm{Aut}(M)$-invariant Borel probability measure on $S_x(M)$. Let $\mathfrak{M}_x(M)$ be the space of $\mathrm{Aut}(M)$-invariant Borel probability measures on $S_x(M)$. Then, there is a unique Borel probability measure $\mathfrak{m}$ on $\mathfrak{M}_x(M)$ such that for any $\mathcal{L}(M)$-formula $\phi(x,a)$,
\[\mu([\phi(x,a)])=\int_{\mathrm{Erg}_x(M)} \nu([\phi(x,a)]) \mathrm{d}\mathfrak{m}(\nu),\]
where $\mathrm{Erg}_x(M)$ is the space of $\mathrm{Aut}(M)$-ergodic measures.

Hence, if there is countable structure $\mathcal{M}\prec\mathbb{M}$ containing $a$ such that $\phi(x,a)$ has measure zero for any ergodic measure on $S_x(M)$, then $\phi(x,a)$ is universally measure zero.
\end{corollary}

\begin{remark} When $M$ is countable, every Borel measure $\mu$ on $S_x(M)$ is regular \cite[Theorem 7.1.7]{Boga}. From \cite[Corollary 1.2]{EvansTsk}, we actually know that when $\mathcal{M}$ is the countable model of an $\omega$-categorical theory,  $\mathrm{Erg}_x(M)$ is closed in the space $\mathfrak{M}_x(M)$.
\end{remark}


For conciseness of notation, in subsequent sections we shall generally refer to $\mu$ as an ergodic measure on $\mathcal{M}^{eq}$ (in the variable $x$). By this we mean that $\mu$ is a Borel probability measure on $S_x(M^{eq})$ invariant under the action of $\mathrm{Aut}(M^{eq})$, which is also $\mathrm{Aut}(M^{eq})$-ergodic. As noted above, when $x$ is a variable in the real sort, there is a one-to-one correspondence between these ergodic measures and the ergodic $\mathrm{Aut}(M)$-invariant Borel probability measures on $S_x(M)$. More generally, given the various correspondences explained in this section we can safely use $\mu$ to denote both the Keisler measure and the corresponding Borel probability measures on $S_x(M)$ and $S_x(M^{eq})$.

\section{Weak algebraic independence and probabilistic independence}\label{WAI}

Given the Borel probability space $(S_x(M^{eq}), \mathcal{B}_x(M^{eq}), \mu)$ induced by an invariant Keisler measure on $M$ in the variable $x$, we wish to study the associated Hilbert space $L^2(\mu)$. In this section we introduce some of the relevant tools for this, following the discussion of \cite{JahelT} and \cite{unitsk}. Furthermore, we shall show how various results of \cite{JahelT} yield that a weak form of algebraic independence between parameters implies a form of probabilistic independence in an ergodic measure. This result is Corollary \ref{probind}.\\

Let $\mathcal{M}$ be a countable $\omega$\Hyphdash* categorical structure. We have that $\mathrm{Aut}(M)$ is a Polish group, that is, a 
 separable completely metrizable topological group. In particular, in the topology of $G=\mathrm{Aut}(M)$, the pointwise stabilizers of finite sets $G_A=\mathrm{Aut}(M/A)$ for $A\subset M$ finite, are neighbourhoods of the identity and the set of cosets of these pointwise stabilizers form a basis of open sets for the topology \cite{Homogeneous}[\S 4.1].

\begin{definition} Let $\mathcal{H}$ be a complex Hilbert space and $U(\mathcal{H})$ be its unitary group. A \textbf{unitary representation} of a topological group $G$ is a continuous action of $G$ on $\mathcal{H}$ by unitary operators. Equivalently, we may say that it is a
an homomorphism $\pi: G\to U(\mathcal{H})$ such that for each $f\in\mathcal{H}$ the map $\tau\mapsto \pi(\tau)\cdot f$ is continuous.
\end{definition}

As noted in the previous section, a Keisler measure $\mu$ on $\mathcal{M}$ in the variable $x$ induces a Borel probability space $(S_x(M^{eq}), \mathcal{B}_x(M^{eq}), \mu)$. We are interested in studying the complex $L^2$-space  $L^2(S_x(M^{eq}), \mathcal{B}_x(M^{eq}), \mu)$, which we abbreviate $L^2(\mu)$. An element $f\in L^2(\mu)$ stands for an equivalence class of complex measurable functions agreeing almost everywhere. Furthermore, since $L^2(\mu)$ is a Hilbert space, it is equipped with an inner product:
\[\text{for } f, g\in L^2(\mu), \langle f, g\rangle:=\int_{S_x(M^{eq})} f\cdot \overline{g} \hspace{0.3em} \mathrm{d}\mu.\]

The action of $\mathrm{Aut}(M)$ on $M$ naturally induces an action $\lambda: \mathrm{Aut}(M)\times L^2(\mu)\to L^2(\mu)$, where for $\sigma\in \mathrm{Aut}(M), f\in L^2(\mu)$, 
\[\lambda(\sigma, f)(p)=f(\sigma^{-1}(p)) \text { for all } p\in S_x(M).\]
This action is well defined and preserves integrals. For each $\sigma\in \mathrm{Aut}(M)$, the map $\Lambda_\sigma:L^2(\mu)\to L^2(\mu)$ given by $f\mapsto \lambda(\sigma, f)$ is a unitary operator. Moreover, the map $\pi: \mathrm{Aut}(M) \to U(L^2(\mu))$ from the group of automorphisms of $M$ to the group of unitary operators of $L^2(\mu)$ given by $\sigma\mapsto \Lambda_\sigma$ is a unitary representation (cf. $\S 3$ of \cite{unitsk}).\\

Let $\mathcal{H}$ be a complex Hilbert space and the action of $G=\mathrm{Aut}(M)$ on $\mathcal{H}$ be a \textbf{unitary representation}. For $A\subseteq M^{eq}$ we write, following \cite{JahelT},
\[\mathcal{H}_A=\overline{\{f\in \mathcal{H} | G_{A'}\cdot f=f \text{ for some finite } A'\subseteq A\}},\]
where for $S\subset\mathcal{H}$, $\overline{S}$ is its closure.

\begin{definition} Let $\mathcal{H}$ be a Hilbert space and $\mathcal{K}\subseteq \mathcal{H}$ be a subspace. We denote the orthogonal complement of $\mathcal{K}$ in $\mathcal{H}$ (often denoted as $\mathcal{K}^\perp$) as $\mathcal{H}\ominus \mathcal{K}$. When $\mathcal{K}$ is a closed subspace, we have that $\mathcal{H}=\mathcal{K}\oplus (\mathcal{H}\ominus\mathcal{K})$. Given subspaces $\mathcal{H}_0, \mathcal{H}_1, \mathcal{H}_2\subseteq \mathcal{H}$, and $\mathcal{H}_0\subseteq\mathcal{H}_i$ for $i\in\{1,2\}$ we say that $\mathcal{H}_1$ and $\mathcal{H}_2$ are \textbf{orthogonal} over $\mathcal{H}_0$ and write $\mathcal{H}_1\perp_{\mathcal{H}_0}\mathcal{H}_2$ if $(\mathcal{H}_1\ominus\mathcal{H}_0)\perp(\mathcal{H}_2\ominus\mathcal{H}_0)$.
\end{definition}
Note that when the $\mathcal{H}_i$ are all closed subspaces for $i\in\{1, 2, 3\}$ with $\mathcal{H}_1\perp_{\mathcal{H}_0}\mathcal{H}_2$, we have that for $f\in \mathcal{H}_1$ we can decompose $f$ into $f=f_1+f_0$, where $f'\in \mathcal{H}_1\ominus\mathcal{H}_0$ and $f_0\in \mathcal{H}_0$. Similarly, for $g\in \mathcal{H}_2$ we can decompose it as $g=g_2+g_0$. Now, by the orthogonalities $(\mathcal{H}_i\ominus\mathcal{H}_0)\perp \mathcal{H}_0$ for $i\in\{1,2\}$ and $(\mathcal{H}_1\ominus\mathcal{H}_0)\perp(\mathcal{H}_2\ominus\mathcal{H}_0)$, we have that
\[\langle f, g \rangle=\langle f_0, g_0 \rangle,\]
where $\langle \cdot , \cdot \rangle$ is the inner product on $\mathcal{H}$.

In an $\omega$\Hyphdash* categorical structure, we have the following theorem of Jahel and Tsankov, which translates weak algebraic independence into orthogonality of the associated Hilbert spaces:

\begin{theorem}\label{deeptheorem} \cite[Theorem 3.2]{JahelT} Let $M$ be $\omega$\Hyphdash* categorical and $G=\mathrm{Aut}(M)$. Let $A, B\subseteq M^{eq}$ be algebraically closed with respect to $\mathrm{acl}^{eq}$. Then, $\mathcal{H}_A\perp_{\mathcal{H}_{A\cap B}}\mathcal{H}_B$.
\end{theorem}

\begin{remark} Recently, \cite{PieceInt} develops some results of this kind outside the context of $\omega$\Hyphdash* categorical structures.
\end{remark}

Note that the space of constant functions is closed in $L^2(\mu)$, hence, from Proposition \ref{erginv} we get: 
\begin{corollary}\label{invfun} Let $\mu$ be an ergodic measure on $\mathcal{M}^{eq}$ in the variable $x$. Then, $L^2(\mu)_{\mathrm{dcl}^{eq}(\emptyset)}$ is generated by the constant indicator function $\mathbb{1}$.
\end{corollary}
\begin{proof} Suppose $f\in L^2(\mu)_{\mathrm{dcl}^{eq}(\emptyset)}$. Then, for some finite $C\subseteq \mathrm{dcl}^{eq}(\emptyset)$, $f$ is $G_C$-invariant. However, since any element in $\mathrm{dcl}^{eq}(\emptyset)$ is fixed by $\mathrm{Aut}(M^{eq})$-automorphisms, $G_C=G_\emptyset$. This means that $f$ is invariant almost everywhere, and therefore constant by ergodicity of the measure.
\end{proof}

An ergodic meausure also concentrates on an orbit. So, for ergodic $\mu$, in $L^2(\mu)$ the constant function $\mathbb{1}$ will be in the same equivalence class as $\mathbb{1}_\phi$, where $\phi$ isolates one of the finitely many types over the empty set, by Ryll-Nardzewski.

\begin{definition}
Let $A, B, C \subseteq \mathcal{M}^{eq}$. Then, $A$ and $C$ are \textbf{weakly algebraically independent} over $B$, written $A\indep{B}{a} C$ if $\mathrm{acl}^{eq}(AB)\cap\mathrm{acl}^{eq}(BC)=\mathrm{acl}^{eq}(B)$.
\end{definition}

Theorem \ref{deeptheorem} yields very powerful results when considering an ergodic measure and the inner product on $L^2(\mu)$.

\begin{theorem}\label{productlemma} Let $\mathcal{M}$ be an $\omega$\Hyphdash* categorical countable structure with $\mathrm{acl}^{eq}(\emptyset)=\mathrm{dcl}^{eq}(\emptyset)$. Consider the complex Hilbert space $\mathcal{H}=L^2(\mu)$, where $\mu$ is an ergodic measure. Suppose that $A, B \subseteq M^{eq}$ are algebraically closed and that $A\indep{}{a} B$. Let $f\in \mathcal{H}_A, g\in\mathcal{H}_B$. Then, 
\[\langle f, g \rangle=\langle f, \mathbb{1}\rangle \overline{\langle g, \mathbb{1}\rangle}.\]
\end{theorem}
\begin{proof}
Let $f_0$ and $g_0$ be the projections of $f$ and $g$ respectively on $\mathcal{H}_\emptyset$. Now, since $f\perp_{\mathcal{H}_\emptyset} g$, $\langle f, g \rangle=\langle f_0, g_0 \rangle$. Since $f_0, g_0\in\mathcal{H}_\emptyset$, by Lemma \ref{invfun}, $f_0=\alpha \mathbb{1}$, $g=\beta \mathbb{1}$ for $\alpha, \beta\in\mathbb{C}$. Now, this yields that 
\[\langle f, g\rangle=\langle f_0, g_0 \rangle=\langle \alpha \mathbb{1},\beta \mathbb{1} \rangle=\alpha \overline{\beta}\langle \mathbb{1}, \mathbb{1}\rangle.\]
Being in a probability space, $\langle \mathbb{1}, \mathbb{1}\rangle=1$. However, since $\langle f-f_0, f_0\rangle=0$, we obtain that $\langle f, \mathbb{1} \rangle =\alpha$, and similarly, $\langle g, \mathbb{1} \rangle =\beta$. This yields the desired result.
\end{proof}

This is already substantially observed in Corollary 3.5 and Remark 3.6 of \cite{JahelT}. In particular, since we are interested in the measures of formulas, we have that


\begin{corollary}\label{probind} Let $\mathcal{M}$ be $\omega$\Hyphdash* categorical. Let $\mu$ be an ergodic measure on ${M}^{eq}$ in the variable $x$. Suppose that $\mathrm{acl}^{eq}(\emptyset)=\mathrm{dcl}^{eq}(\emptyset)$. Let $a, {b}$ be tuples from $\mathcal{M}^{eq}$ such that ${a}\indep{}{a}{b}$. Then, for any $\mathcal{L}^{eq}$-formulas $\phi(x, {y}), \psi(x, {z})$, 
\[\mu(\phi(x, {a}) \wedge \psi(x, {b}))=\mu(\phi(x, {a}) )\mu(\psi(x, {b})).\]

\end{corollary}
\begin{proof}
Note that the indicator functions $\mathbb{1}_{\phi(x, {a})}$ and $\mathbb{1}_{\psi(x, {b})}$ are in $\mathcal{H}_{\mathrm{acl}^{eq}({a})}$ and $\mathcal{H}_{\mathrm{acl}^{eq}({b})}$ respectively, and that
\begin{align*}
    \mu(\phi(x, {a})\cap \psi(x, {b})) & = \int_{S_x(M)} \mathbb{1}_{\phi(x, {a})\cap\psi(x, {b})} \mathrm{d}\mu\\ & =\int_{S_x(M)} \mathbb{1}_{\phi(x, {a})}\cdot \mathbb{1}_{\psi(x, {b})} \mathrm{d}\mu \\ & = \left(\int_{S_x(M)} \mathbb{1}_{\phi(x, A)}\mathrm{d}\mu\right)\overline{\left(\int_{S_x(M)} \mathbb{1}_{\phi(x, {a})}\mathrm{d}\mu\right)} \\
    &=\mu(\phi(x, {a}))\overline{\mu(\psi(x, {b}))} \\ &= \mu(\phi(x, {a}))\mu(\psi(x, {b})).
\end{align*}
Here the third equality holds by Theorem \ref{productlemma} and the last one follows since our measure is real-valued. Hence, the result follows. 
\end{proof}

\begin{remark}\label{converse} This corollary has a partial converse which is well known in ergodic theory \cite[Proposition 4.8]{JahelT}. Fix an invariant Keisler measure $\mu$ and suppose that for any formula $\phi(x,a)$ there is an automorphism $\sigma\in\mathrm{Aut}(M)$ such that 
\[\mu(\phi(x,a)\wedge \phi(x, \sigma \cdot a))=\mu(\phi(x,a))^2.\]
Then, $\mu$ is ergodic. 
\end{remark}

We conclude this section with a brief discussion of the assumption of $\mathrm{acl}^{eq}(\emptyset)=\mathrm{dcl}^{eq}(\emptyset)$ in Theorem \ref{productlemma} and Corollary \ref{probind}. The assumption is needed in order to have the equality  $L^2(\mu)_{\mathrm{acl}^{eq}(\emptyset)}= L^2(\mu)_{\mathrm{dcl}^{eq}(\emptyset)}$. The latter then yields the desired results since we know from Corollary \ref{invfun} that $L^2(\mu)_{\mathrm{dcl}^{eq}(\emptyset)}$ is generated by the constant indicator function. However, we can also obtain the relevant independence results when $\mathrm{acl}^{eq}(\emptyset)\supsetneq\mathrm{dcl}^{eq}(\emptyset)$. The following lemma is common knowledge:


\begin{lemma}\label{finacl} Let $\mathcal{M}$ be $\omega$\Hyphdash* categorical. Let $A\subseteq M$ be finite. Fix the variable $x$ in the home sort. Then, there is $A_0\subseteq \mathrm{acl}^{eq}(A)$ finite such that for every $b$, a tuple from $\mathcal{M}$ in the variable $x$,
    \begin{equation}\label{finacleqn}
        \mathrm{tp}(b/A_0)\vdash \mathrm{tp}(b/\mathrm{acl}^{eq}(A)).
    \end{equation}
By this we mean that if $b$ and $b'$ in the variable $x$ have the same type over $A_0$, then they have the same type over $\mathrm{acl}^{eq}(A)$ in $\mathcal{M}^{eq}$.

\end{lemma}
\begin{proof}
This is substantially Lemma 2.4 in \cite{EvansTsk}.
\end{proof}

A useful consequences of the lemma is:

\begin{corollary}\label{largera} Let $\mathcal{M}$ be $\omega$\Hyphdash* categorical. Let $\mu$ be an invariant Keisler measure on $\mathcal{M}^{eq}$ in the variable $x$. Then, there is finite $A_0\subseteq \mathrm{acl}^{eq}(\emptyset)$ such that 
\[L^2(\mu)_{A_0}=L^2(\mu)_{\mathrm{acl}^{eq}(\emptyset)}.\]
\end{corollary}
\begin{proof} Let $A_0\subseteq\mathrm{acl}^{eq}(\emptyset)$ be finite such that \ref{finacleqn} from Lemma \ref{finacl} holds. Suppose that $f\in L^2(\mu)$ is $G_A$-invariant for $G=\mathrm{Aut}(M^{eq})$ and $A$ a finite subset of $\mathrm{acl}^{eq}(\emptyset)$. Then, $f$ is also $G_{A'}$-invariant for $A'=A\cup A_0$. But then, by choice of $A_0$ (with respect to the variable $x$), the $G_{A'}$-orbit of any type in $S_x(M^{eq})$ is the same as its $G_{A_0}$-orbit. Hence, $f$ is $G_{A_0}$-invariant.
\end{proof}

For an $\omega$\Hyphdash* categorical structure $\mathcal{M}$ and given the variable $x$ and $A_0$ as above, there are finitely many types over $A_0$ in the variable $x$, isolated by formulas $\chi_1(x), \dots, \chi_m(x)$. By additivity, any invariant Keisler measure $\mu$ can be written as a weighted sum of measures $\mu_{\chi_i}$ for $1\leq i\leq m'\leq m$, where, for $\mu(\chi_i(x))>0$, $\mu_{\chi_i}$ is the $\mathrm{Aut}(M/A_0)$-invariant Keisler measure induced from $\mu$ by
\[\mu_{\chi_i}(\phi(x,a))=\frac{\mu(\phi(x,a)\wedge \chi_i(x))}{\mu(\chi_i(x))}.\]
From Corollary \ref{largera}, we know that for any $\mathrm{Aut}(M/A_0)$-ergodic measure $\nu$ we have that for $a\indep{}{a} b$, and $\mathcal{L}^{eq}$-formulas $\phi(x,y), \psi(x, z)$,
\[\nu(\phi(x,a)\wedge \psi(x, b))=\nu(\phi(x,a))\nu(\psi(x, b)).\]
Hence, in the context of $\mathrm{acl}^{eq}(\emptyset)\supsetneq\mathrm{dcl}^{eq}(\emptyset)$, in order to study invariant Keisler measures on an $\omega$\Hyphdash* categorical structure $\mathcal{M}$ we may naturally move to study the $\mathrm{Aut}(M/A_0)$-invariant Keisler measures. 

\section{A strong independence theorem}\label{SIT}
In this section, we prove that in a simple $\omega$\Hyphdash* categorical structure $\mathcal{M}$ with $\mathrm{acl}^{eq}(\emptyset)=\mathrm{dcl}^{eq}(\emptyset)$ if forking over the emptyset is the same as being universally measure zero, then $\mathcal{M}$ satisfies a stronger version of the independence theorem over finite algebraically closed sets. We conclude with some consequences for $\omega$\Hyphdash* categorical $MS$-measurable structures.\\

Firstly, we note how the measure $\mu(\phi(x, a)\wedge \psi(x, b))$ only depends on the types of the individual parameters when $a$ and $b$ are weakly algebraically independent.

\begin{corollary}\label{aindepinv}
Let $\mathcal{M}$ be an $\omega$\Hyphdash* categorical structure with $\mathrm{acl}^{eq}(\emptyset)=\mathrm{dcl}^{eq}(\emptyset)$. Suppose that $a, b\in \mathcal{M}^{eq}$ are such that $a\indep{}{a} b$. Let $\phi(x, y), \psi(x, z)$ be $\mathcal{L}^{eq}$-formulas. Then, for an arbitrary $\mathrm{Aut}(\mathcal{M})$-invariant Keisler measure $\mu: \mathrm{Def}_x(\mathcal{M})\to[0,1]$,
\[\mu(\phi(x, a)\wedge \psi(x, b))\]
only depends on $\mathrm{tp}(a)$ and $\mathrm{tp}(b)$.
\end{corollary}
\begin{proof}
Let $a'\indep{}{a} b'$ be such that $a'\equiv a$ and $b'\equiv b$. Then, by Corollary \ref{probind}, we get that for any ergodic measure $\nu$,
\[\nu(\phi(x, a)\wedge \psi(x, b))=\nu(\phi(x, a))\nu(\psi(x, b))=\nu(\phi(x, a'))\nu(\psi(x, b'))=\nu(\phi(x, a')\wedge \psi(x, b')).\]
But then, by the ergodic decomposition, Corollary \ref{decomp}, we have that for any $\mathrm{Aut}(M)$-invariant Keisler measure
\[\mu(\phi(x, a)\wedge \psi(x, b))=\mu(\phi(x, a')\wedge \psi(x, b')).\]

\end{proof}

\begin{definition} Let $\mathbb{M}$ be a monster model and $A\subset \mathbb{M}$ be a small subset. We write $F(A)$ for the set of formulas with parameters from $\mathbb{M}$ forking over $A$.
\end{definition}

As noted in the introduction, from \cite{lots_of_authors} we know that $F(\emptyset)\subseteq \mathcal{O}(\emptyset)$. We are interested in studying what happens when $F(\emptyset)=\mathcal{O}(\emptyset)$ in order to find a structure where the two sets differs. In an $\omega$\Hyphdash* categorical context, by $\omega$-saturation, nothing is lost by considering formulas with parameters from the countable model $\mathcal{M}$. In fact, if $F(\emptyset)\subsetneq \mathcal{O}(\emptyset)$, this will be witnessed by a formula with parameters from $\mathcal{M}$. To see this, recall the discussion in Remark \ref{noloss}.

\begin{definition}\label{AIT} Let $\mathcal{M}$ be an $\mathcal{L}$-structure. We say that $\mathcal{M}$ \textbf{satisfies the strong independence theorem} over $A\subseteq \mathcal{M}^{eq}$ if the following holds:

Let $a, b, c_0, c_1\in M^{eq}$ be such that $a\indep{A}{a} b$, $c_0\equiv_A c_1$ and $c_0\indep{A}{}a$,  $c_1\indep{A}{}b$. Then, there is $c^*\in \mathcal{M}^{eq}$ such that $c^*\equiv_{Aa} c_0, c^*\equiv_{Ab} c_1$, and $c^*\indep{A}{}ab$.
\end{definition}
Here, the relation $\indep{}{}$ denotes non-forking independence. From Corollary \ref{aindepinv} we obtain that simple $\omega$\Hyphdash* categorical structures where forking coincides with being universally measure zero satisfy the strong independence theorem over the empty set. 

\begin{theorem}\label{strongITT} Let $\mathcal{M}$ be  a simple $\omega$\Hyphdash* categorical structure with $\mathrm{acl}^{eq}(\emptyset)=\mathrm{dcl}^{eq}(\emptyset)$. Suppose that  $F(\emptyset)=\mathcal{O}(\emptyset)$, i.e. a formula forks over the empty set if and only if it is universally measure zero. Then, $\mathcal{M}$ satisfies the strong independence theorem over the empty set. 
\end{theorem}
\begin{proof}
Suppose that there are $a, b, c_0, c_1\in \mathcal{M}^{eq}$ as in Definition \ref{AIT}. Let $\phi(x, a)$ and $\psi(x, b)$ isolate $\mathrm{tp}(c_0/a)$ and $\mathrm{tp}(c_1/b)$. By the existence property of non-forking independence, there is $b'\equiv b$ such that $b'\indep{}{}a$. By Corollary \ref{aindepinv}, for any Keisler measure
\[\mu(\phi(x, a)\wedge \psi(x, b))=\mu(\phi(x, a)\wedge \psi(x, b')).\]
By simplicity, $\phi(x, a)\wedge \psi(x, b')$ does not fork over the empty set since the independence theorem holds over $\emptyset$. Hence, by $F(\emptyset)=\mathcal{O}(\emptyset)$,
$\phi(x, a)\wedge \psi(x, b)$ is not universally measure zero, and so is non-forking over the empty set. This proves the strong independence theorem over the empty-set for $\mathcal{M}$.
\end{proof}

\begin{remark} In general, a simple $\omega$\Hyphdash* categorical theory with $\mathrm{acl}^{eq}(\emptyset)=\mathrm{dcl}^{eq}(\emptyset)$ satisfies the independence theorem over the empty set (see, for example, \cite{me}). However, the condition of $F(\emptyset)=\mathcal{O}(\emptyset)$ yields a strengthening of the independence theorem, where the "base" of the amalgamation is weakly algebraically independent rather than non-forking independent. In fact, while non-forking independence implies weak algebraic independence, the converse does not hold in general.

\begin{definition} We say that an $\omega$\Hyphdash* categorical simple structure $\mathcal{M}$ is \textbf{one-based} if given $A, B\subseteq \mathcal{M}^{eq}$ algebraically closed, then, 
\[A\indep{A\cap B}{} B.\]
\end{definition}
If $\mathcal{M}$ is one-based, for $A\subseteq M$ we have that
\[b\indep{A}{a} c \text{ if and only if } b\indep{A}{} c.\]
In particular, in a one-based structure, satisfying the independence theorem over $A$ is equivalent to satisfying the strong independence theorem over $A$.\\

However, there are not one-based simple $\omega$\Hyphdash* categorical structures. The only known example of this are $\omega$\Hyphdash* categorical Hrushovski constructions. For these, to satisfy the strong independence theorem is a genuinely stronger requirement than satisfying the independence theorem.
\end{remark}

We conclude this section with some consequences for $MS$-measurable structures in an $\omega$\Hyphdash* categorical context. For a general introduction to $MS$-measurable structures we suggest \cite{EM} or the original article \cite{MS}. In \cite{me}, I discuss in detail $\omega$\Hyphdash* categorical $MS$-measurable structures, and we direct the reader to that article for the relevant definitions and some basic results. The general idea is that $MS$-measurable structures have an associated dimension-measure function assigning each definable set a dimension and a measure. The dimension-measure function is invariant (and definable), the measure always takes strictly positive values and the dimension and the measure satisfy Fubini's theorem. Being $MS$-measurable is a property of a theory. An important feature proved in \cite{MS} is that $\mathcal{M}$ is $MS$-measurable if and only if $\mathcal{M}^{eq}$ is.

\begin{lemma}\label{MSF=O} Let $\mathcal{M}$ be an $\omega$\Hyphdash* categorical $MS$-measurable structure. Then, $F(A)=\mathcal{O}(A)$ for any finite $A\subseteq M$.
\end{lemma}
\begin{proof}
From \cite{me}, we can take the dimension part of the dimension-measure to be $SU$-rank. Suppose there is any formula $\phi(x,b)$ which does not fork over $A$, but is universally measure zero for any $\mathrm{Aut}(M/A)$-invariant Keisler measure. Let $c\vDash \phi(x,b)$. By non-forking independence, $SU(c/Ab)=SU(c/A)$. Let $\psi(x,A)$ isolate $\mathrm{tp}(c/A)$. By $MS$-measurability we have an induced $\mathrm{Aut}(M/A)$-invariant Keisler measure $\mu_\psi$ on $\mathrm{Def}_x(M)$ given by
\[\mu_\psi(\chi(x,d))=
\begin{cases}
\frac{\mu(\chi(x,d)\wedge \psi(x,A))}{\mu(\psi(x,A))} & \text{ for } SU(\chi(x,d)\wedge \psi(x,A)))=SU(\psi(x,A)),\\
0 & \text{otherwise}.
\end{cases}
\]
where $\mu$ is the measure in the dimension-measure. Being universally measure zero, we have that $\mu_\psi(\phi(x,b))=\mu(\phi(x,b))=0$, which contradicts positivity of the measure. 
\end{proof}

From Corollary \ref{aindepinv} we obtain that $\omega$\Hyphdash* categorical $MS$-measurable structures satisfy the strong independence theorem over the algebraic closures of finite sets. Indeed, we also get the corresponding probabilistic independence statement. Recall from Lemma \ref{finacl} that for $A\subseteq M^{eq}$ the algebraic closure of a finite set, and fixing the variable $x$, there is $A_0\subset A$ finite such that types in the variable $x$ over $A$ are isolated by $\mathcal{L}^{eq}$-formulas with parameters from $A_0$.

\begin{theorem} Let $\mathcal{M}$ be $\omega$\Hyphdash* categorical and $MS$-measurable. Then, it satisfies the strong independence theorem over the algebraic closures of finite sets:\\
Let $A\subseteq M^{eq}$ be the algebraic closure of a finite set, $a_0, a_1, b, c$ tuples from $M^{eq}$. Suppose $a_0\equiv_A a_1$, $b\indep{A}{a}c$ and that $a_0\indep{A}{} b$, $a_0\indep{A}{} c$. Then, there is $a^*$ such that $a^*\equiv_{Ab} a_0$, $a^*\equiv_{Ac} a_1$, and $a^*\indep{A}{}bc$. Moreover, 
\[\mu(\mathrm{tp}(a_0/Ab)\cup \mathrm{tp}(a_1/Ac))=\frac{\mu(a_0/Ab)\mu(a_1/Ac)}{\mu(a_0/A)}.\]
\end{theorem}
\begin{proof} Let $\phi(x, b)$ and $\psi(x, c)$ isolate $\mathrm{tp}(a_0/Ab)$ and $\mathrm{tp}(a_1/Ac)$ respectively. Let $\chi(x)$ isolate $\mathrm{tp}(a_i/A)$, and let $\mu_\chi$ be the $\mathrm{Aut}(M^{eq}/A)$ invariant Keisler measure induced by $\chi(x)$. We consider $\mu_\chi(\phi(x, b)\wedge \psi(x, c))$. By extension we can find $c'\equiv_A c$ such that $c'\indep{A}{}b$. Since non-forking independence implies weak algebraic independence, by Corollary \ref{aindepinv}
\[\mu_\chi(\phi(x, b)\wedge \psi(x, c'))=\mu_\chi(\phi(x, b)\wedge \psi(x, c)).\]
But the former has positive measure by the independence theorem over algebraically closed sets (and $d$-independence being the same as non-forking independence \cite{me}). Hence, $\phi(x, b)\wedge \psi(x, c)$ must have a realisation which is independent from $bc$ over $A$. Moreover, by \cite{me}, for $b\indep{A}{}c'$ we have that 
\[\mu(\phi(x, b)\wedge \psi(x, c))=\mu(\phi(x, b)\wedge \psi(x, c'))=\frac{\mu(\phi(x, b))\mu(\psi(x, c))}{\mu(\chi(x))}.\]
This yields the desired equation.
\end{proof}

\begin{remark}
By the proof above and Remark \ref{converse}, an $MS$-measurable $\omega$\Hyphdash* categorical structure induces various ergodic Keisler measures. Let $A\subseteq M^{eq}$ be the algebraic closure of a finite set and $\chi(x)$ isolate a type over $A$. Then, the induced $\mathrm{Aut}(M/A)$-invariant Keisler measure $\mu_\chi$ is $\mathrm{Aut}(M/A)$-ergodic. We can also prove this using \cite{me} and Remark \ref{converse}.
\end{remark}

\section{The counterexample}\label{counterex}
We are now ready to introduce our example of a simple $\omega$\Hyphdash* categorical structure with $F(\emptyset)\subsetneq \mathcal{O}(\emptyset)$. It will be an $\omega$\Hyphdash* categorical Hrushovski construction. For details on Hrushovski constructions the reader may refer to \cite[$\S$ 6.2.1]{Wagner:ST}. This construction is also used in \cite{me} as an example of an $\omega$\Hyphdash* categorical supersimple structure which is not $MS$-measurable, so the reader may refer to that article for details on this particular construction. The basic idea is that (non-trivial) simple $\omega$\Hyphdash* categorical Hrushovski constructions are not one-based. Hence, there are pairs which are weakly algebraic independent, but forking-dependent. This allows us to build Hrushovski constructions which are simple but which do not satisfy the strong independence theorem. 

\begin{theorem}\label{constr} There is an $\omega$\Hyphdash* categorical graph $\mathcal{M}$ supersimple of $SU$-rank $2$ with the following properties:
\begin{itemize}
\item $\mathrm{Aut}(M)$ acts transitively on the vertices of $M$.
    \item Points are algebraically closed;
    \item Edges are algebraically closed, but for $b\in M$ the formula $E(x,b)$ asserting that $x$ has an edge with $b$ forks over $\emptyset$;
    \item For $a, b\in M$ with no edge between them, $\mathrm{acl}(ab)=ab$ or $\mathrm{acl}(ab)=abc$, where $abc$ is a path of length two with endpoints $a$ and $b$. In either case, $a\indep{}{}b$;
    \item The smallest $k$-cycle in $\mathcal{M}$ is a $6$-cycle;
    \item The structure $\mathcal{M}$ has weak elimination of imaginaries;
\end{itemize}
Furthermore, we can choose $\mathcal{M}$ to satisfy independent $n$-amalgamation over the algebraic closures of finite sets for any $n\in\mathbb{N}$, and also for all $n\in\mathbb{N}$.
\end{theorem}
\begin{proof}
The structure is an $\omega$\Hyphdash* categorical Hrushovski construction. It is the same construction as in \cite[Construction 5.1]{me}. In the appendix of that article we prove supersimplicity, weak elimination of imaginaries and note how higher independent amalgamation can be obtained. The other properties also follow by construction and basic calculations with the dimension, recalling that in a Hrushovski construction, the Hrushovski dimension corresponds to $SU$-rank \cite[Corollary 6.2.26]{Wagner:ST}.
\end{proof}

\begin{remark} The graph $\mathcal{M}$ in Theorem \ref{constr} is also \textbf{extremely amenable} in the sense of \cite{FOAmen}. That is, every type in $S(\emptyset)$ extends to an $\mathrm{Aut}(M)$-invariant type over $M$. Since invariant types can be considered as invariant Keisler measures taking only the values $0$ and $1$, this guarantees that there are some invariant Keisler measures on $\mathcal{M}$ in each variable. To see that these structures are extremely amenable, for a finite tuple $\overline{a}$ from $M$, consider the type $p(\overline{x})$ given by
\[\bigcup_{B\subset M \text{ finite}}\left\{\mathrm{tp}(\overline{a}'/B) \mid \overline{a}'\equiv\overline{a}, \mathrm{acl}(\overline{a}B)=\mathrm{acl}(\overline{a})\cup\mathrm{acl}(B), \neg E(c,b) \text{ for }c\in\mathrm{acl}(\overline{a}), b\in\mathrm{acl}(B)\right\}.\]
Substantially, $p(\overline{x})$ asserts that $\overline{x}$ has no relation to $\mathcal{M}$ and is weakly algebraically independent from it. From  \cite[Section 4]{me} we can conclude that this type is consistent, complete and invariant by the extension property and because types of finite tuples are determined by the quantifier-free types of their algebraic closures.
\end{remark}

It is commonly known that in a structure with weak elimination of imaginaries, for $A, B\subseteq M$ algebraically closed in $\mathcal{M}$, we have that
\[\mathrm{acl}^{eq}(A)\cap\mathrm{acl}^{eq}(B)=\mathrm{acl}^{eq}(A\cap B).\]
Hence, for $a$ and $b$ sharing an edge, $a\indep{}{a} b$, but $\mathrm{tp}(a/b)$ forks over $\emptyset$. Meanwhile, for $a$ and $c$ at distance two from each other, $a\indep{}{}c$. Hence, for $\mathcal{M}$ to satisfy the strong independence theorem over the empty set, $\mathcal{M}_f$ should include $5$-cycles. But we have built $\mathcal{M}_f$ to exclude these. This implies the existence of a formula which does not fork over the empty-set, but is universally measure zero. We can give this formula explicitly:

\begin{theorem} Let $\mathcal{M}$ be the $\omega$\Hyphdash* categorical structure described in Theorem \ref{constr}. Let $\phi(x, y)$ be the formula stating that the points $x$ and $y$ are exactly at distance two from each other. Then, for $a\in M$, the formula $\phi(x, a)$ is universally measure zero but does not fork over the emptyset. 
\end{theorem}
\begin{proof}
Let $a, b\in M$ share an edge. Now, $\phi(x, a)\wedge \phi(x, b)$ is inconsistent since $\mathcal{M}$ avoids $5$-cycles. Meanwhile, $\phi(x, a)$ does not fork over the empty set since $c\indep{}{}a$ for $c$ at distance two from $a$. However, since $a\indep{}{a} b$, for an ergodic measure,
\[0=\mu(\phi(x, a)\wedge \phi(x, b))=\mu(\phi(x, a))\mu(\phi(x, b))=\mu(\phi(x, a))^2,\]
where the second equality follows by Corollary \ref{probind} and the last equality by transitivity of the action of $\mathrm{Aut}(M)$. From the calculation above we get that  $\mu(\phi(x, a))=0$ for any ergodic measure. By Corollary \ref{decomp}, $\phi(x,a)$ is universally measure zero. 
\end{proof}

Hence, we get the desired counterexample:

\begin{theorem}\label{mainthm} There are $\omega$\Hyphdash* categorical supersimple theories $T$ of finite $SU$-rank with a formula $\phi(x, a)$ which does not fork over the empty set, but which is universally measure zero. 
\end{theorem}

Furthermore, by Lemma \ref{MSF=O}, we get a counterexample to the question of Elwes and Macpherson \cite{EM}:

\begin{corollary} There are $\omega$\Hyphdash* categorical supersimple theories $T$ of finite $SU$-rank which are not $MS$-measurable.
\end{corollary}

The above was also answered in \cite{Measam} and \cite{me}, with $\omega$\Hyphdash* categorical Hrushovski constructions as counterexamples.

\begin{remark}
The point at the heart of our proof is that it is possible to build simple $\omega$\Hyphdash* categorical Hrushovski constructions which do not satisfy the strong independence theorem. This can be done in different relational languages and with different ranks for the final structure. For example, using a $3$-hypergraph, we can build a simple $\omega$\Hyphdash* categorical Hrushovski construction of $SU$-rank 1 with non-forking formulas which are universally measure zero. 
\end{remark}

\subsection{Conclusions}

Our result makes substantial use of the fact we work with not one-based $\omega$\Hyphdash* categorical structures. In fact, simple $\omega$\Hyphdash* categorical structures with $\mathrm{acl}^{eq}(\emptyset)=\mathrm{dcl}^{eq}(\emptyset)$ satisfy the independence theorem over the empty set. And in a one-based structure, satisfying this is equivalent to satisfying the strong independence theorem. Our counterexample relies on our simple structure not satisfying the strong independence theorem. And for this to be possible we must work with a not one-based structure. This raises two natural questions:

\begin{question}\label{Q1} Suppose that $\mathcal{M}$ is a simple $\omega$\Hyphdash* categorical structure satisfying the strong independence theorem over the algebraic closures of finite sets. Does this imply that $F(\emptyset)=\mathcal{O}(\emptyset)$?
\end{question}

\begin{question}\label{Q2} Suppose that $\mathcal{M}$ is a simple one-based $\omega$\Hyphdash* categorical structure. Does this imply that $F(\emptyset)=\mathcal{O}(\emptyset)$?
\end{question}

For both questions we may also ask whether the hypotheses mentioned (substituting "simple" by "supersimple") imply $MS$-measurability. Indeed, it would be interesting to know whether supersimple $\omega$\Hyphdash* categorical one-based structures are $MS$-measurable. The finite rank assumption is not needed since supersimple $\omega$\Hyphdash* categorical one-based structures are of finite rank \cite{WagEv}. For Question \ref{Q1}, we might also ask whether satisfying some sufficiently strong higher amalgamation property implies $F(\emptyset)=\mathcal{O}(\emptyset)$. In this article we have shown that satisfying $n$- independent amalgamation for all $n\in\mathbb{N}$ over finite algebraically closed sets is not sufficient, but an adequate generalisation of the strong independence theorem might work.\\

It is also relevant to note that satisfying the strong independence theorem over the algebraic closures of finite sets is a weaker condition than one-basedness. In fact, one can prove that the standard example of not-one based $\omega$\Hyphdash* categorical Hrushovski construction (e.g. as in \cite[Example 6.2.27]{Wagner:ST}) does satisfy the strong independence theorem over finite algebraically closed sets. In fact, we strongly suspect that another $\omega$\Hyphdash* categorical Hrushovski construction may yield an example of a structure with $F(\emptyset)\subsetneq \mathcal{O}(\emptyset)$ while satisfying the strong independence theorem (and arbitrarily strong higher amalgamation properties).\\




 \printbibliography

\end{document}